\RequirePackage{fix-cm}
\documentclass[smallextended]{svjour3}       
\smartqed  
\usepackage{graphicx}
\usepackage[utf8]{inputenc}
\usepackage{dsfont,xspace,amsmath,amssymb}
\usepackage[marginclue,langtrack]{fixme} 
\usepackage[autostyle=true]{csquotes}
\newcommand{\Rbb}{\ensuremath{\mathds{R}}}
\newcommand{\set}[1]{\ensuremath{\left\{#1\right\}}}
\newcommand{\setval}[2]{\set{#1, \ldots, #2}}
\newcommand{\eg}{e.\,g.\xspace}%
\newcommand{\Nbb}{\ensuremath{\mathds{N}}}
\newcommand{\Ocal}{\ensuremath{\mathcal{O}}}
\newcommand{\ie}{i.\,e.\xspace}%
\DeclareMathOperator{\Proj}{Proj}
\DeclareMathOperator{\argmin}{argmin}

\journalname{Mathematical Programming}
\begin{document}

\title{The quadratic assignment problem: the linearization of Xia and Yuan is weaker than the linearization of Adams and Johnson
}
\subtitle{and a family of cuts to narrow the gap}

\titlerunning{The QAP: a comparison of the linearizations of Xia\&Yuan and Adams\&Johnson}        

\author{Christine Huber \and
		Wolfgang~F.~Riedl
}


\institute{Wolfgang~F.~Riedl \at	
           Universität der Bundeswehr München\\
					 Professur für Statistik, insb. Risikomanagement\\
					 Werner-Heisenberg-Weg 39\\
					 D - 85577 Neubiberg\\
					 Germany\\
              \email{wf.riedl@unibw.de}        
}

\date{Received: date / Accepted: date}

\maketitle

\begin{abstract}
The quadratic assignment problem is a well-known optimization problem with numerous applications. A common strategy to solve it is to use one of its linearizations and then apply the toolbox of mixed integer linear programming methods. One measure of quality of a mixed integer formulation is the quality of its linear relaxation.

In this paper, we compare two linearizations of the quadratic assignment problem and prove that the linear relaxation of the linearization of Adams and Johnson is contained in the linear relaxation of the linearization of Xia and Yuan. We furthermore develop a Branch and Cut approach using the insights obtained in the proof that enhances the linearization of Xia and Yuan via a new family of cuts called $ab$-cuts.

\keywords{quadratic assignment problem\and linearization\and lift and project\and linear relaxation\and cutting planes}
\subclass{MSC 52-B05 \and MSC 05-C38}
\end{abstract}

\section{The quadratic assignment problem}

The quadratic assignment problem is a well-known optimization problem which has a long history, numerous applications and has obtained broad coverage in the literature (see, \eg, \cite{MR2267435} for a survey, or \cite{MR1924955,MR1290346,MR2238659,MR3084089,MR2979433,nyberg2014some,MR2373098,MR2913063,MR2314426} for recent approaches) since its introduction in \cite{MR0089106}. Its aim is to find an optimal assignment $\Phi: \setval{1}{n} \rightarrow\setval{1}{n}$ with respect to a quadratic objective function $\sum_{i,k = 1}^n p_{ik}d_{\Phi(i)\Phi(k)} +\sum_{i = 1}^np_id_{\Phi(i)}$, where $p_{ik}\in\Rbb, d_{jl}\in\Rbb\;\forall j,l\in\setval{1}{n}$ and $d_i\in\Rbb\;\forall i\in\setval{1}{n}$. A straightforward ILP formulation with non-linear objective function is presented in \eqref{eq:ilp1} -- \eqref{eq:ilp4} using binary variables $x_{ij}\in\set{0,1}\;\forall i,j\in\setval{1}{n}$ with $x_{ij}=1 \Leftrightarrow \Phi(i)=j\;\forall i,j\in\setval{1}{n}$.

\begin{align}
	\label{eq:ilp1}
	\min \sum_{i,j,k,l = 1}^n p_{ik}d_{jl}x_{ij}x_{kl} &+\sum_{i,j = 1}^np_id_jx_{ij}\\
	\text{s.t.}\quad\sum_{i=1}^n x_{ij}&= 1\quad\forall j\in\setval{1}{n}\\
	\sum_{j=1}^n x_{ij}&= 1\quad\forall i\in\setval{1}{n}\\
	\label{eq:ilp4}
	x_{ij}&\in\set{0,1}\quad\forall i,j\in\setval{1}{n}
\end{align}

For convenience, the underlying set of all permutations of the set $\setval{1}{n}$ will be summarized in the permutation matrix
\begin{align*}
	X_n = \Bigg\{x_{ij}\in\set{0,1}:&\sum_{i=1}^n x_{ij}= 1\quad\forall j\in\setval{1}{n}\\
	&\sum_{j=1}^n x_{ij}= 1\quad\forall i\in\setval{1}{n}\Bigg\}.
\end{align*}
If we relax the integrality requirement, we denote this by
\begin{align*}
	LX_n = \Bigg\{x_{ij}\in\Rbb_{\geq 0}:&\sum_{i=1}^n x_{ij}= 1\quad\forall j\in\setval{1}{n}\\
	&\sum_{j=1}^n x_{ij}= 1\quad\forall i\in\setval{1}{n}\Bigg\}.
\end{align*}
The objective function's coefficients will be denoted by matrices 
\begin{equation*}
	D = (d_{jl})_{j,l\in\setval{1}{n}}\text{ and } P = (p_{ik})_{i,k\in\setval{1}{n}}.
\end{equation*}

To solve the Quadratic Assignment Problem, a variety of methods can be used. One common approach is to use an equivalent linear formulation of the problem and then apply the vast toolbox of Mixed Integer Linear Programming methods. Many linearizations of the QAP have been proposed \cite{MR1290346,MR2238659,KAUFMAN1978207,MR3084089,MR2913063,MR2314426}. Two of these which are frequently used will be presented in more detail in Sec.~\ref{sec:linearizations}. In Sec.~\ref{sec:comparison} we will compare the quality of both linearizations with respect to their linear relaxation. We will use a Lift-And-Project approach to prove that one of them is superior to the other. Using the theoretical insights obtained in the proof, we will present $ab$-cuts in Sec.~\ref{sec:newInequalities}: a new family of valid inequalities for the quadratic assignment problem that strengthen the smaller linearization.

\begin{remark}[Notation]
	For the remainder of this paper, we will denote elements to be assigned by $i$ and $k$, which will be assigned to elements denoted by $j$ and $l$.
\end{remark}

\begin{remark}[Omitted combinations]
	Note that for the objective function 	
	\begin{equation*}
		\sum_{i,j,k,l = 1}^n p_{ik}d_{jl}x_{ij}x_{kl} +\sum_{i,j = 1}^np_id_jx_{ij},
	\end{equation*}
	where $x\in X_n$, 
	\begin{align*}
		x_{ij}x_{il} &= 0\quad\forall i,j,l\in\setval{1}{n}&
		x_{ij}x_{kj} &= 0\quad\forall i,j,k\in\setval{1}{n}
	\end{align*}
	We therefore assume $i\neq k$ and $j\neq l$ for all combinations $i,j,k,l \in\setval{1}{n}$ used in the following.
\end{remark}

\section{Linearizations of the QAP}
\label{sec:linearizations}
In the following section, linearizations of the quadratic assignment problem are presented. We will present two of them in detail: the linearization of Adams and Johnson and the linearization of Xia and Yuan. 

\subsection{The linearization of Adams and Johnson}

The linearization of Adams and Johnson was introduced in \cite{MR1290346}. It replaces the objective function's nonlinear term $x_{ij}x_{kl}$ by a new variable $y_{ijkl}$ and can be regarded as the result of a Lift-And-Project procedure. 

To ensure the relationship 
\begin{equation}
	y_{ijkl} = x_{ij}x_{kl}\quad\forall i,j,k,l\in\setval{1}{n},
	\label{eq:yxRelationship}
\end{equation}
three families of linear constraints are introduced. The resulting MILP has $\Ocal(n^4)$ continuous variables, $\Ocal(n^2)$ binary variables and $\Ocal(n^4)$ constraints.

\begin{align}
\allowdisplaybreaks
	\label{eq:aj1}
	\min\sum_{\substack{i,j,k,l = 1\\k\neq i, l\neq j}}^n p_{ik}d_{jl}y_{ijkl} & +\sum_{i,j = 1}^n p_id_jx_{ij}\\
	\label{eq:aj2}
	\text{s.t.}\quad \sum_{\substack{i=1\\i\neq k}}^n y_{ijkl}&=x_{kl}\quad\forall j,k,l \in\setval{1}{n}, j\neq l\\
	\label{eq:aj3}
	\sum_{\substack{j=1\\j\neq l}}^n y_{ijkl}&=x_{kl}\quad\forall i,k,l \in\setval{1}{n}, i\neq k\\
	\label{eq:aj4}
	y_{ijkl}&=y_{klij}\quad\forall i,j,k,l \in\setval{1}{n}, k\neq i, l\neq j\\
	\label{eq:aj5}
	y_{ijkl}&\geq 0\quad\forall i,j,k,l \in\setval{1}{n}, k\neq i, l\neq j\\
	\label{eq:aj6}
	x&\in X_n
\end{align}
For coefficient matrices $D, P\in\Rbb^{n\times n}$, we will denote \eqref{eq:aj1} -- \eqref{eq:aj6} by $AJ(D,P)$, and its linear relaxation where $x\in LX_n$ by $LAJ(D,P)$.

Albeit its increase in variable and constraint size, the linearization is often used due to its good bounds.

\subsection{The linearization of Xia and Yuan}

A different kind of linearization was introduced by \cite{MR2238659}, based on the previous work of \cite{KAUFMAN1978207}. By factorizing the nonlinear objective function, they substitute the term $x_{ij} \cdot \sum_{k,l = 1}^n p_{ik}d_{jl}x_{kl} = z_{ij}$, which yields the objective function $\sum_{i,j=1}^nz_{ij}$. 

To ensure the equation 
\begin{equation}
	\label{eq:XYSubstitution}
	x_{ij} \cdot \sum_{\substack{k,l = 1\\k\neq i, l\neq j}}^n p_{ik}d_{jl}x_{kl} = z_{ij}\quad\forall i,j\in\setval{1}{n}
\end{equation}
one uses two \enquote{Big-M} constraints with constants 
\begin{equation}
	u_{ij} = \max_{x\in X_{n-1}}\sum_{\substack{k,l = 1\\k \neq i, l\neq j}}^n p_{ik}d_{jl}x_{kl}\quad\forall i,j\in\setval{1}{n}
\label{eq:DefinitionU}
\end{equation}
and
\begin{equation}
	l_{ij} = \min_{x\in X_{n-1}}\sum_{\substack{k,l = 1\\k\neq i, l\neq j}}^n p_{ik}d_{jl}x_{kl}\quad\forall i,j\in\setval{1}{n}.
\label{eq:DefinitionL}
\end{equation}
Then 
\begin{equation*}
	z_{ij}\geq \sum_{\substack{k,l=1\\k\neq i, l\neq j}}^np_{ik}d_{jl}x_{kl}+u_{ij}(x_{ij}-1)\quad\forall i,j\in\setval{1}{n}
\end{equation*}
ensures the validity of \eqref{eq:XYSubstitution} if $x_{ij} = 1$. To strengthen the linearization's linear relaxation, \cite{MR2238659} additionally introduced 
\begin{equation*}
	z_{ij}\geq l_{ij}x_{ij}\quad\forall i,j \in\setval{1}{n}.
\end{equation*}
The resulting linearization only uses $\Ocal(n^2)$ continuous and binary variables as well as $\Ocal(n^2)$ constraints. It is therefore smaller than the linearization of Adams and Johnson, however its linear relaxation usually is also worse.
\begin{align}
	\label{eq:XY1}
	\min\sum_{i,j = 1}^nz_{ij} &+\sum_{i,j=1}^np_id_jx_{ij}\\
	\label{eq:XYlInequality}
	\text{s.t.}\quad z_{ij}&\geq l_{ij}x_{ij}\quad\forall i,j \in\setval{1}{n}\\
	z_{ij}&\geq \sum_{\substack{k,l=1\\k\neq i, l\neq j}}^np_{ik}d_{jl}x_{kl}+u_{ij}(x_{ij}-1)\quad\forall i,j \in\setval{1}{n}\\
	\label{eq:XY4}
	x&\in X_n
\end{align}
We denote \eqref{eq:XY1} -- \eqref{eq:XY4} for given coefficients $D, P\in\Rbb^{n\times n}$ by $XY(D,P)$, and the corresponding linear relaxation where $x\in LX_n$ by $LXY(D,P)$.

\subsection{Further linearizations}

Different other linearizations were proposed for the QAP. Many approaches substituted the quadratic term $x_{ij}x_{kl}$ by additional variables $x_{ijkl}$ and used different constraints to ensure equivalence of the obtained formulation with the original one, such as \cite{MR0152361} with
\begin{align*}
	\min\sum_{\substack{i,j,k,l = 1\\k\neq i, l\neq j}}^n p_{ik}d_{jl}y_{ijkl} & +\sum_{i,j = 1}^n p_id_jx_{ij}\\
	\text{s.t.}\quad \sum_{\substack{i,j,k,l = 1\\k\neq i, l\neq j}}^n y_{ijkl}&=n^2\\
	x_{ij}+x_{kl}-2y_{ijkl}&\geq 0\quad\forall i,j,k,l \in\setval{1}{n}, i\neq k,j\neq l\\
	y_{ijkl}&\in\set{0,1}\quad\forall i,j,k,l \in\setval{1}{n}, k\neq i, l\neq j\\
	x&\in X_n
\end{align*}
and \cite{MR678819} with
\begin{align*}
	\min\sum_{\substack{i,j,k,l = 1\\k\neq i, l\neq j}}^n p_{ik}d_{jl}y_{ijkl} & +\sum_{i,j = 1}^n p_id_jx_{ij}\\
	\text{s.t.}\quad \sum_{\substack{i,j = 1\\i\neq k, j\neq l}}^n y_{ijkl}&=n x_{kl}\quad\forall k,l\in\setval{1}{n}\\
	\sum_{\substack{k,l = 1\\k\neq i, l\neq j}}^n y_{ijkl}&=n x_{ij}\quad\forall i,j\in\setval{1}{n}\\
	0\leq y_{ijkl}&\leq 1\quad\forall i,j,k,l \in\setval{1}{n}, k\neq i, l\neq j\\
	x&\in X_n
\end{align*} and 
{\allowdisplaybreaks
\begin{align*}
	\min\sum_{\substack{i,j,k,l = 1\\k\neq i, l\neq j}}^n p_{ik}d_{jl}y_{ijkl} & +\sum_{i,j = 1}^n p_id_jx_{ij}\\
	\text{s.t.}\quad \sum_{\substack{i = 1\\i\neq k}}^n y_{ijkl}&=x_{kl}\quad\forall j,k,l\in\setval{1}{n}, l\neq j\\
	\sum_{\substack{j = 1\\j\neq l}}^n y_{ijkl}&=x_{kl}\quad\forall i,k,l\in\setval{1}{n}, k\neq i\\
	\sum_{\substack{k = 1\\k\neq i}}^n y_{ijkl}&=x_{kl}\quad\forall i,j,l\in\setval{1}{n}, l\neq j\\
	\sum_{\substack{l = 1\\l\neq j}}^n y_{ijkl}&=x_{kl}\quad\forall i,j,k\in\setval{1}{n}, k\neq i\\
	0\leq y_{ijkl}&\leq 1\quad\forall i,j,k,l \in\setval{1}{n}, k\neq i, l\neq j\\
	x&\in X_n
\end{align*}}

The linearization of Adams and Johnson presented above dominates these linearizations with respect to its linear relaxation (see \cite{MR1290346}). The Lift-And-Project concept used to obtain the linearization of Adams and Johnson can be applied further to obtain an even tighter linearization using six indices (see \cite{MR2314426}) with substitution $t_{ijklpq} = x_{ij}x_{kl}x_{pq}$.

Considering smaller linearizations, \cite{KAUFMAN1978207} introduced the first one equivalent to the linearization of Xia and Yuan without constraint \eqref{eq:XYlInequality}. Consequently, its linear relaxation was weak. This was improved by the aforementioned linearization of Xia and Yuan, who investigated different possible values for $l_{ij}$ and $u_{ij}$ and also proved a method to efficiently compute $l_{ij}$ and $u_{ij}$ when defined as in \eqref{eq:DefinitionU} and \eqref{eq:DefinitionL}.

\section{A comparison of both linearizations}
\label{sec:comparison}
In the following section, we will compare the two aforementioned linearizations of Adams and Johnson as well as Xia and Yuan with regard to the quality of their linear relaxations. To do so, we will use a Lift-And-Project approach as presented by \cite{MR2194735}. After a short recap of the process (for the corresponding proofs we refer to \cite{MR2194735}), we will apply it to the linearizations. Subsequently, we will prove this section's main result: the linear relaxation of the linearization of Adams and Johnson $LAJ(D,P)$ will be proven to be better than the one of the linearization of Xia and Yuan $LXY(D,P)$.

\subsection{Theory on Lift-And-Project to compare linear formulations}

Assume two linear formulations $P$ and $Q$ to the same problem, $P$ being defined on variables $z, x$ and $Q$ on variables $y, x$. If there is an affine relationship between variables $z$ and $y$ of the form $z = Ty + e$ with suitable matrix $T$ and vector $e$, we can compare both formulations. To compare both polyhedra, we can lift one of them (\eg $Q$) in a common variable space $(z, x, y)$ and then project it to the space of the second ($P$). 

Assume $x\in \Rbb^p, y\in\Rbb^q$ and $z\in\Rbb^m$ and let $X\subseteq \Rbb^p$. Let furthermore 
\begin{equation*}
	P = \set{(x,z)\in\Rbb^{p+m}: Ax + Bz \leq a, x\in X}
\end{equation*}
where $A\in\Rbb^{l_1\times p}, B\in\Rbb^{l_1\times m}, a\in\Rbb^{l_1}$ for some $l_1 \in\Nbb$ and 
\begin{equation*}
	Q = \set{(x,y)\in\Rbb^{p+q}: Cx + Dy = b, Ey = c, Fy \leq d, x\in X}
\end{equation*}
where $C\in\Rbb^{l_2\times p}, D\in\Rbb^{l_2\times q}, b\in\Rbb^{l_2}, E\in\Rbb^{l_3\times q}, c\in\Rbb^{l_3}, F\in\Rbb^{l_4\times q}, d\in\Rbb^{l_4}$ for $l_2,l_3,l_4\in\Nbb$ and assume that there exist $T\in \Rbb^{m\times q}, e\in\Rbb^m$ such that
\begin{equation*}
	z = Ty + e.
\end{equation*}

An extended formulation $Q^+$ of $Q$ in variables space $(x, y, z)$ is
\begin{align*}
	Q^+ = \left\{(x,y,z)\right.&\in\Rbb^{p+q+m}:\\
	-Ty + z&= e\\
	Cx + Dy&= b\\
	Ey&= c\\
	Fy&\leq d\\
	x&\left.\in X\right\}.
\end{align*}
To project $Q^+$ down into variable space $(x,z)$, we need the corresponding projection cone
\begin{align*}
	W = \left\{(v, w^1, w^2, w^3)\right.&\in\Rbb^{m+l_2+l_3+l_4}:\\
	vT&=w^1D + w^2 E + w^3 F\\
	w^3&\left.\geq 0\right\}.
\end{align*}

Then the projection of $Q^+$ into $(x,z)$ is 
\begin{align*}
	\Proj_{x,z}\left(Q^+\right) = \left\{(x,z)\right.&\in\Rbb^{p+m}:\\
	vz + w^1Cx&\leq ve + w^1b + w^2 c + w^3 d \quad\forall (v,w^1,w^2,w^3)\in W\\
	x&\left.\in X\right\}.
\end{align*}

\subsection{Comparison of $LAJ(D,P)$ and $LXY(D,P)$}

We can now apply the above theory to the two linearizations of the quadratic assignment problem. Let $D, P\in\Rbb^{n\times n}$ coefficient matrices, and consider
\begin{align*}
	LXY(D,P) = \left\{(x,z)\right.&\in \Rbb^{2n^2}:\\
	-z_{ij}+l_{ij}x_{ij}&\leq 0\quad\forall i,j \in\setval{1}{n}\\
	-z_{ij}+u_{ij}x_{ij}+\sum_{\substack{k,l=1\\k\neq i, l\neq j}}^np_{ik}d_{jl}x_{kl}&\leq u_{ij}\quad\forall i,j\in\setval{1}{n}\\
	x&\left.\in LX_n\right\}
\end{align*}
the linearization of Xia and Yuan and
\begin{align*}
	LAJ(D,P) = \left\{(x,y)\right.&\in\Rbb^{n^2+n^4}:\\
	\sum_{\substack{i=1\\i\neq k}}^ny_{ijkl}-x_{kl}&=0\quad\forall j,k,l\in\setval{1}{n}, j\neq l\\
	\sum_{\substack{j=1\\j\neq l}}^ny_{ijkl}-x_{kl}&=0\quad\forall i,k,l\in\setval{1}{n}, i\neq k\\
	y_{ijkl}-y_{klij}&=0\quad\forall i,j,k,l\in\setval{1}{n}, k\neq i, l\neq j\\
	-y_{ijkl}&\leq 0\quad\forall i,j,k,l\in\setval{1}{n}, k\neq i, l\neq j\\
	x&\left.\in LX_n\right\}
\end{align*}
the linearization of Adams and Johnson. From \eqref{eq:XYSubstitution} and \eqref{eq:yxRelationship} we can derive the linear relationship
\begin{equation}
\label{eq:zyzusammenhang}
	z_{ij} = \sum_{\substack{k,l=1\\k\neq i, l\neq j}}^np_{ik}d_{jl}y_{ijkl}\quad\forall i,j\in\setval{1}{n}
\end{equation}
and thus obtain an extended formulation of the linearization of Adams and Johnson (with dual multipliers in brackets)
\begin{align*}
\allowdisplaybreaks
	LAJ(D,P)^+ = \left\{(x,y,z)\right.&\in\Rbb^{2n^2+n^4}:\\
	z_{ij} - \sum_{\substack{k,l=1\\k\neq i, l\neq j}}^np_{ik}d_{jl}y_{ijkl}&=0\quad\forall i,j\in\setval{1}{n}&(\alpha_{ij})\\
	\sum_{\substack{i=1\\i\neq k}}^ny_{ijkl}-x_{kl}&=0\quad\forall j,k,l\in\setval{1}{n}, j\neq l&(\beta^1_{jkl})\\
	\sum_{\substack{j=1\\j\neq l}}^ny_{ijkl}-x_{kl}&=0\quad\forall i,k,l\in\setval{1}{n}, i\neq k&(\beta^2_{ikl})\\
	y_{ijkl}-y_{klij}&=0\quad\forall i,j,k,l\in\setval{1}{n}, k\neq i, j\neq l&(\gamma_{ijkl})\\
	-y_{ijkl}&\leq 0\quad\forall i,j,k,l\in\setval{1}{n}, k\neq i, j\neq l\\
	x&\left.\in LX_n\right\}
\end{align*}

Using denoted multipliers $\alpha_{ij}\in\Rbb\;\forall i,j\in\setval{1}{n}$, $\beta^1_{jkl}\in\Rbb\;\forall j,k,l\in\setval{1}{n}, j\neq l$, $\beta^2_{ikl}\in\Rbb\;\forall i,k,l\in\setval{1}{n}, i\neq k$ and $\gamma_{ijkl}\in\Rbb\;\forall i,j,k,l\in\setval{1}{n}, k\neq i, l\neq j$, we yield the projection cone
\begin{align*}
	W = \left\{(\alpha,\beta^1,\beta^2,\gamma)\right.&\in\Rbb^{n^4+2n^3+n^2}:\\
	p_{ik}d_{jl}\alpha_{ij}\leq \beta^1_{jkl} &\left.+ \beta^2_{ikl} + \gamma_{ijkl}-\gamma_{klij}\quad\forall i,j,k,l \in\setval{1}{n}, k\neq i, l\neq j\right\}.
\end{align*}
One can simplify this to
\begin{align}
	W = \left\{(\alpha,\beta^1\right.&,\beta^2,\gamma)\in\Rbb^{n^4+2n^3+n^2}:\notag\\
	\label{eq:proj_cone:1}
	p_{ik}d_{jl}\alpha_{ij}&\leq \beta^1_{jkl} + \beta^2_{ikl} + \gamma_{ijkl}\quad\forall i,j,k,l \in\setval{1}{n}, k\neq i, l\neq j\\
	\label{eq:proj_cone:2}
	\gamma_{ijkl}&\left.=-\gamma_{klij}\quad\forall i,j,k,l\in\setval{1}{n}, k\neq i, l\neq j\right\}
\end{align}
with resulting projection of the extended formulation to variable space $(x,z)$
\begin{align*}
	\Proj_{x,z}\left(LAJ(D,P)^+\right) = \left\{(x,z)\right.&\in\Rbb^{2n^2}:\\
	\sum_{i,j=1}^n\alpha_{ij}z_{ij} - \sum_{i,j=1}^n \left(\sum_{\substack{l=1\\l\neq j}}^n \beta^1_{lij} + \sum_{\substack{k=1\\k\neq i}}^n\beta^2_{kij}\right)x_{ij}&\leq 0\quad\forall (\alpha,\beta^1,\beta^2,\gamma)\in W\\
	x&\left.\in LX_n\right\}.
\end{align*}

One imminent question is now whether one can conclude that one of both linearizations has a tighter linear relaxation. And in fact, one can:
\begin{theorem}
	Let $D, P\in\Rbb^{n\times n}$ coefficient matrices, and denote by $LXY(D,P)$ the linear relaxation of the Xia-Yuan linearization, and by $LAJ(D,P)$ the linear relaxation of the Adams-Johnson linearization. Then
	\begin{equation*}
		\Proj_{x,z}\left(LAJ(D,P)^+\right) \subseteq LXY(D,P).
	\end{equation*}
	\label{theorem:AJbesserXY}
\end{theorem}

To prove Thm.~\ref{theorem:AJbesserXY} we will show that both defining inequalities of the Xia-Yuan linearization can be derived from the projected Adams-Johnson linearization. Therefore we conclude that the latter is tighter.

\begin{proposition}
	\label{prop:inequalityLContained}
	Let $a,b \in\setval{1}{n}$. Inequality
	\begin{equation*}
		z_{ab}\geq l_{ab}x_{ab}
	\end{equation*}
	with $l_{ab}$ as defined in \eqref{eq:DefinitionL} is part of the description of $\Proj_{x,z}\left(LAJ(D,P)^+\right)$.
\end{proposition}

\begin{proof}
	The statement to prove is equivalent to the proposition that there exist $(\alpha, \beta^1, \beta^2, \gamma)\in W$ such that
	\begin{equation*}
		\sum_{i,j=1}^n\alpha_{ij}z_{ij} - \sum_{i,j=1}^n \left(\sum_{\substack{l=1\\l\neq j}}^n \beta^1_{lij} + \sum_{\substack{k=1\\k\neq i}}^n\beta^2_{kij}\right)x_{ij}\leq 0
	\end{equation*}
	is equivalent to
	\begin{equation*}
		z_{ab}\geq l_{ab}x_{ab}.
	\end{equation*}
	
	The multipliers in question are
	\begin{itemize}
		\item $\alpha_{ij} = 
			\begin{cases}
				-1&\text{if }(i,j) = (a,b)\\
				0&\text{otherwise}
			\end{cases}
			\quad\forall i,j \in\setval{1}{n}$
		\item $\beta^1_{lij} = 0\quad\forall l,i,j\in\setval{1}{n}, l\neq j, (i,j) \neq (a,b)$
		\item $\beta^2_{kij} = 0\quad\forall k,i,j\in\setval{1}{n}, k\neq i, (i,j) \neq (a,b)$
		\item $\gamma_{ijkl} = 
			\begin{cases}
				0&\text{if }(i,j) \neq (a,b), (k,l) \neq (a,b)\\
				-p_{ak}d_{bl}&\text{if }(i,j) = (a,b)\\
				p_{ai}d_{bj}&\text{if }(k,l) = (a,b)\\
			\end{cases}\quad
			\begin{aligned}
				\forall i,j,k,l\in\setval{1}{n}\\
				k\neq i, l\neq j
			\end{aligned}$
		\item $\beta^1_{lab}, \beta^2_{kab}$ where
			\begin{align}
			\label{eq:multipliersLByDual}
				\begin{aligned}
					-p_{ak}d_{bl}&\leq \beta^1_{lab}+\beta^2_{kab}\quad\forall k,l\in\setval{1}{n}, k\neq a, l\neq b\\
					-l_{ab}&=\sum_{\substack{l=1\\l\neq b}}^n\beta^1_{lab} + \sum_{\substack{k=1\\k\neq a}}^n\beta^2_{kab}
				\end{aligned}
			\end{align}
	\end{itemize}
	The existence of $\beta^1_{lab}, \beta^2_{kab}$ such that \eqref{eq:multipliersLByDual} is valid will be shown in Lemma~\ref{lemma:ExistenceMultipliersLByDual}. $(\alpha,\beta^1,\beta^2,\gamma)$ as defined are contained in the projection cone, as shown by verifying \eqref{eq:proj_cone:1} for all cases (equation \eqref{eq:proj_cone:2} is valid by inspection):
	\begin{enumerate}
		\item Let $(i,j)=(a,b), k\neq a, l\neq b$:
			\begin{equation*}
				p_{ak}d_{bl}\alpha_{ab} = -p_{ak}d_{bl} = 0 + 0 -p_{ak}d_{bl} = \beta^1_{bkl} + \beta^2_{akl} + \gamma_{abkl}
			\end{equation*}
		\item Let $(k,l) = (a,b), i\neq k, j\neq l$:
			\begin{equation*}
				p_{ia}d_{jb}\alpha_{ij} = 0 = -p_{ai}d_{bj} + p_{ai}d_{bj} \leq \beta^1_{jab} + \beta^2_{iab} + \gamma_{ijab}
			\end{equation*}
		\item Let $(i,j)\neq(a,b), (k,l) \neq (a,b)$:
			\begin{equation*}
				p_{ik}d_{jl}\alpha_{ij} = 0 = 0 + 0 + 0 = \beta^1_{jkl} + \beta^2_{ikl} + \gamma_{ijkl}
			\end{equation*}
	\end{enumerate}
	Furthermore the resulting inequality is equivalent to the desired constraint of the linearization of Xia and Yuan:
	\begin{align*}
		\sum_{i,j=1}^n\alpha_{ij}z_{ij} - \sum_{i,j=1}^n \left(\sum_{\substack{l=1\\l\neq j}}^n \beta^1_{lij} + \sum_{\substack{k=1\\k\neq i}}^n\beta^2_{kij}\right)x_{ij} & =\\
		-z_{ab} - \left(\sum_{\substack{l=1\\l\neq b}}^n \beta^1_{lab} + \sum_{\substack{k=1\\k\neq a}}^n\beta^2_{kab}\right)x_{ab} & =\\
		-z_{ab} + l_{ab}x_{ab} &\leq 0
	\end{align*}
\end{proof}

The existence of coefficients as required in \eqref{eq:multipliersLByDual} can be verified as follows.
\begin{lemma}
	\label{lemma:ExistenceMultipliersLByDual}
	Let $a,b\in\setval{1}{n}$ and $l_{ab}$ as defined in \eqref{eq:DefinitionL}. There exist multipliers $\beta^1_{lab}\in\Rbb\;\forall l\in\setval{1}{n}, l\neq b$ and $\beta^2_{kab}\in\Rbb\;\forall k\in\setval{1}{n}, k\neq a$ such that
	\begin{align*}
		-p_{ak}d_{bl}&\leq \beta^1_{lab}+\beta^2_{kab}\quad\forall k,l\in\setval{1}{n}, k\neq a, l\neq b\\
		-l_{ab}&=\sum_{\substack{l=1\\l\neq b}}^n\beta^1_{lab} + \sum_{\substack{k=1\\k\neq a}}^n\beta^2_{kab}
	\end{align*}
\end{lemma}
\begin{proof}
	Let $a,b\in\setval{1}{n}$ and consider the following linear program with associated dual variables:
	\begin{align*}
		\min \sum_{\substack{k,l=1\\k\neq a, l\neq b}}^n p_{ak}d_{bl}x_{kl}&\\
		\text{s.t.}\quad\sum_{\substack{k=1\\k\neq a}}^nx_{kl}&=1\quad\forall l\in\setval{1}{n}, l\neq b &(\beta^1_{lab})\\
		\sum_{\substack{l=1\\l\neq b}}^nx_{kl}&=1\quad\forall k\in\setval{1}{n}, k\neq a &(\beta^2_{kab})\\
		x_{kl}&\geq 0\quad\forall k,l\in\setval{1}{n}, k\neq a, l\neq b
	\end{align*}
	Its dual has the form
	\begin{align*}
		\max \sum_{\substack{l=1\\l\neq b}}^n\beta^1_{lab} + \sum_{\substack{k=1\\k\neq a}}^n\beta^2_{kab} &\\
		\text{s.t.}\quad\beta^1_{lab} + \beta^2_{kab}&\leq p_{ak}d_{bl}\quad\forall k,l\in\setval{1}{n}, k\neq a, l\neq b
	\end{align*}
	The primal problem is a linear assignment problem with total unimodular matrix, thus there exists an optimal integral solution with value (by definition)~$l_{ab}$. The corresponding optimal dual solution therefore satisfies
	\begin{align*}
		\sum_{\substack{l=1\\l\neq b}}^n\beta^1_{lab} + \sum_{\substack{k=1\\k\neq a}}^n\beta^2_{kab}& = l_{ab}\\
		\quad\beta^1_{lab} + \beta^2_{kab}&\leq p_{ak}d_{bl}\quad\forall k,l\in\setval{1}{n}, k\neq a, l\neq b
	\end{align*}
	An inversion of the dual solution yields the required result.
\end{proof}

The first inequality of the linearization of Xia and Yuan has been shown to be deducable from the projection of Adams and Johnson. For the second inequality the same result can be shown.
\begin{proposition}
	\label{prop:inequalityUContained}
	Let $a,b \in\setval{1}{n}$. Inequality
	\begin{equation*}
		z_{ab}\geq \sum_{\substack{k,l=1\\k\neq a, l\neq b}}^np_{ak}d_{bl}x_{kl}+u_{ab}(x_{ab}-1)
	\end{equation*}
	is contained in the description of $\Proj_{x,z}\left(LAJ(D,P)^+\right)$.
\end{proposition}

\begin{proof}
	As in the proof of Prop.~\ref{prop:inequalityLContained} the statement to prove is equivalent to the proposition that there exist $(\alpha, \beta^1, \beta^2, \gamma)\in W$ such that
	\begin{equation*}
		\sum_{i,j=1}^n\alpha_{ij}z_{ij} - \sum_{i,j=1}^n \left(\sum_{\substack{l=1\\l\neq b}}^n \beta^1_{lij} + \sum_{\substack{k=1\\k\neq a}}^n\beta^2_{kij}\right)x_{ij}\leq 0
	\end{equation*}
	is equivalent to
	\begin{equation*}
		z_{ab}\geq \sum_{\substack{k,l=1\\k\neq a, l\neq b}}^np_{ak}d_{bl}x_{kl}+u_{ab}(x_{ab}-1).
	\end{equation*}

	The multipliers are
	\begin{itemize}
		\item $\alpha_{ij} = 
			\begin{cases}
				-1&\text{if }(i,j) = (a,b)\\
				0&\text{otherwise}
			\end{cases}
			\quad\forall i,j \in\setval{1}{n}$
		\item $\beta^1_{lij} =
			\begin{cases}
				0&\text{if }(i,j) = (a,b)\\
				-\frac{p_{ai}d_{bj}}{2}&\text{if }i\neq a, j \neq b, l = b\\
				0&\text{if }i\neq a, j\neq b, l \neq b\\
				\text{see below}&i = a, j\neq b\\
				\text{see below}&i \neq a, j = b\\
			\end{cases}
			\quad\forall l,i,j\in\setval{1}{n}, l\neq j$
		\item $\beta^2_{kij} = 
			\begin{cases}
				0&\text{if }(i,j) = (a,b)\\
				-\frac{p_{ai}d_{bj}}{2}&\text{if }i\neq a, j \neq b, k = a\\
				0&\text{if }i\neq a, j\neq b, k \neq a\\
				\text{see below}&i = a, j\neq b\\
				\text{see below}&i \neq a, j = b\\
			\end{cases}
			\quad\forall k,i,j\in\setval{1}{n}, k\neq i$
		\item Let $d\neq b$ and $\beta^1_{lad}, \beta^2_{kad}$ where
			\begin{align}
				\label{eq:multipliersUByDual1}
				\begin{aligned}
					\frac{p_{ak}d_{bl}}{2}&\leq \beta^1_{lad}+\beta^2_{kad}\quad\forall k,l\in\setval{1}{n}, k\neq a, l\neq d\\
					\frac{u_{ab}}{2}&=\sum_{\substack{l=1\\l\neq d}}^n\beta^1_{lad} + \sum_{\substack{k=1\\k\neq a}}^n\beta^2_{kad}
				\end{aligned}
			\end{align}
		\item Let $c\neq a$ and $\beta^1_{lcb}, \beta^2_{kcb}$ where
			\begin{align}
				\label{eq:multipliersUByDual2}
				\begin{aligned}
					\frac{p_{ak}d_{bl}}{2}&\leq \beta^1_{lcb}+\beta^2_{kcb}\quad\forall k,l\in\setval{1}{n}, k\neq c, l\neq b\\
					\frac{u_{ab}}{2}&=\sum_{\substack{l=1\\l\neq b}}^n\beta^1_{lcb} + \sum_{\substack{k=1\\k\neq c}}^n\beta^2_{kcb}
				\end{aligned}
			\end{align}
		\item $\gamma_{ijkl} =
			\begin{cases}
				0&\text{if }(i,j) = (a,b)\\
				0&\text{if }(k,l) = (a,b)\\
				-\frac{p_{ai}d_{bj}}{2}&\text{if }k = a, l \neq b\\
				\frac{p_{ak}d_{bl}}{2}&\text{if }i = a, j \neq b\\
				-\frac{p_{ai}d_{bj}}{2}&\text{if }k \neq a, l = b\\
				\frac{p_{ak}d_{bl}}{2}&\text{if }i \neq a, j = b\\
				0&\text{otherwise}
			\end{cases}
			\quad\forall i,j,k,l\in\setval{1}{n}, k\neq i, l\neq j$
	\end{itemize}
	
	For the existence of multipliers for which \eqref{eq:multipliersUByDual1} and \eqref{eq:multipliersUByDual2} hold we refer to Lemma~\ref{lemma:ExistenceMultipliersUByDual}. $(\alpha,\beta^1,\beta^2,\gamma)$ as defined above are contained in the projection cone, as proven by checking constraints \eqref{eq:proj_cone:1} of the projection cone (requirement \eqref{eq:proj_cone:2} is valid by inspection):
	\begin{enumerate}
		\item Let $(i,j) = (a,b), k\neq a, l\neq b$:
			\begin{equation*}
				p_{ak}d_{bl}\alpha_{ab} = -p_{ak}d_{bl} = -\frac{p_{ak}d_{bl}}{2}-\frac{p_{ak}d_{bl}}{2} + 0 = \beta^1_{bkl} + \beta^2_{akl} + \gamma_{abkl}
			\end{equation*}
		\item Let $d \neq b$ and $(i,j) = (a,d), k\neq a, l\neq b, l\neq d$:
			\begin{equation*}
				p_{ak}d_{dl}\alpha_{ad} = 0 = 0 - \frac{p_{ak}d_{bl}}{2} + \frac{p_{ak}d_{bl}}{2} = \beta^1_{dkl} + \beta^2_{akl} + \gamma_{adkl}
			\end{equation*}
		\item Let $c\neq a$ and $(i,j) = (c,b), k\neq a, k\neq c, l\neq b$:
			\begin{equation*}
				p_{ck}d_{bl}\alpha_{cb} = 0 = -\frac{p_{ak}d_{bl}}{2} + 0 +\frac{p_{ak}d_{bl}}{2} = \beta^1_{bkl} + \beta^2_{ckl} + \gamma_{cbkl}
			\end{equation*}
		\item Let $(k,l) = (a,b), i\neq a, j\neq b$:
			\begin{equation*}
				p_{ia}d_{jb}\alpha_{ij} = 0 = 0 + 0 + 0 = \beta^1_{jab} + \beta^2_{iab} + \gamma_{ijab}
			\end{equation*}
		\item Let $d\neq b$ and $(k,l) = (a,d), i\neq a, j\neq b, j\neq d$:
			\begin{equation*}
				p_{ia}d_{jd}\alpha_{ij} = 0 = \frac{p_{ai}d_{bj}}{2} - \frac{p_{ai}d_{bj}}{2} \leq \beta^1_{jad} + \beta^2_{iad} + \gamma_{ijad}
			\end{equation*}
		\item Let $c\neq a$ and $(k,l) = (c,b), i\neq a, i\neq c, j\neq b$:
			\begin{equation*}
				p_{ic}d_{jb}\alpha_{ij} = 0 = \frac{p_{ai}d_{bj}}{2} - \frac{p_{ai}d_{bj}}{2} \leq \beta^1_{jcb} + \beta^2_{icb} + \gamma_{ijcb}
			\end{equation*}
		\item Let $i \neq a, j\neq b, k\neq a, l\neq b:$
			\begin{equation*}
				p_{ik}d_{jl}\alpha_{ij} = 0 = 0 + 0 + 0 = \beta^1_{jkl} + \beta^2_{ikl} + \gamma_{ijkl}
			\end{equation*}
	\end{enumerate}
	
	Furthermore
	\begin{align*}
		\sum_{i,j=1}^n\alpha_{ij}z_{ij} - \sum_{i,j=1}^n \left(\sum_{\substack{l=1\\l\neq j}}^n \beta^1_{lij} + \sum_{\substack{k=1\\k\neq i}}^n\beta^2_{kij}\right)x_{ij} & =\\
		-z_{ab} - 0 x_{ab} - \sum_{\substack{d=1\\d\neq b}}^n\left(\sum_{\substack{l=1\\l\neq d}}^n \beta^1_{lad} + \sum_{\substack{k=1\\k\neq a}}^n\beta^2_{kad}\right)x_{ad}&\\
		- \sum_{\substack{c=1\\c\neq a}}^n\left(\sum_{\substack{l=1\\l\neq b}}^n \beta^1_{lcb} + \sum_{\substack{k=1\\k\neq c}}^n\beta^2_{kcb}\right)x_{cb} - \sum_{\substack{i,j=1\\i\neq a, j\neq b}}^n\left(\sum_{\substack{l=1\\l\neq j}}^n \beta^1_{lij} + \sum_{\substack{k=1\\k\neq i}}^n\beta^2_{kij}\right)x_{ij}&=\\
		-z_{ab} + u_{ab}x_{ab} -\left(\frac{u_{ab}}{2}+\frac{u_{ab}}{2}\right)x_{ab}- \sum_{\substack{d=1\\d\neq b}}^n\left(\frac{u_{ab}}{2}\right)x_{ad}&\\
		 - \sum_{\substack{c=1\\c\neq a}}^n\left(\frac{u_{ab}}{2}\right)x_{cb} - \sum_{\substack{i,j=1\\i\neq a, j\neq b}}^n\left(-\frac{p_{ai}d_{bj}}{2}-\frac{p_{ai}d_{bj}}{2}\right)x_{ij}&=\\
		-z_{ab} + u_{ab}x_{ab} - \sum_{d=1}^n\left(\frac{u_{ab}}{2}\right)x_{ad}&\\
		 - \sum_{c=1}^n\left(\frac{u_{ab}}{2}\right)x_{cb} - \sum_{\substack{i,j=1\\i\neq a, j\neq b}}^n\left(-p_{ai}d_{bj}\right)x_{ij}&=\\
		-z_{ab}+u_{ab}x_{ab} - 2\left(\frac{u_{ab}}{2}\right)+\sum_{\substack{i,j=1\\i\neq a, j\neq b}}\left(p_{ai}d_{bj}\right)x_{ij}&\leq 0
	\end{align*}
\end{proof}

Again, coefficients as required in \eqref{eq:multipliersUByDual1} and \eqref{eq:multipliersUByDual2} exist.
\begin{lemma}
	\label{lemma:ExistenceMultipliersUByDual}
	Let $a,b\in\setval{1}{n}$, $d\neq b, c\neq a$ and consider $u_{ab}$ as defined in \eqref{eq:DefinitionU}. There exist $\beta^1_{lad}\in\Rbb\;\forall l\neq d$ and $\beta^2_{kad}\in\Rbb\;\forall k\neq a$ such that
	\begin{align*}
		\frac{p_{ak}d_{bl}}{2}&\leq \beta^1_{lad}+\beta^2_{kad}\quad\forall k,l\in\setval{1}{n}, k\neq a, l\neq d\\
		\frac{u_{ab}}{2}&=\sum_{\substack{l=1\\l\neq d}}^n\beta^1_{lad} + \sum_{\substack{k=1\\k\neq a}}^n\beta^2_{kad}
	\end{align*}
	There exist $\beta^1_{lcb}\in\Rbb\;\forall l\neq b$ and $\beta^2_{kcb}\in\Rbb\;\forall k\neq c$ such that
	\begin{align*}
		\frac{p_{ak}d_{bl}}{2}&\leq \beta^1_{lcb}+\beta^2_{kcb}\quad\forall k,l\in\setval{1}{n}, k\neq c, l\neq b\\
		\frac{u_{ab}}{2}&=\sum_{\substack{l=1\\l\neq b}}^n\beta^1_{lcb} + \sum_{\substack{k=1\\k\neq c}}^n\beta^2_{kcb}
	\end{align*}
\end{lemma}
\begin{proof}
	Let $a,b\in\setval{1}{n}$ and consider the following linear program with associated dual variables:
	\begin{align*}
		\max \sum_{\substack{k,l=1\\k\neq a, l\neq b}}^n \frac{p_{ak}d_{bl}}{2}x_{kl}&\\
		\text{s.t.}\quad\sum_{\substack{k=1\\k\neq a}}^nx_{kl}&=1\quad\forall l\in\setval{1}{n}, l\neq b &(\beta^1_{lab})\\
		\sum_{\substack{l=1\\l\neq b}}^nx_{kl}&=1\quad\forall k\in\setval{1}{n}, k\neq a &(\beta^2_{kab})\\
		x_{kl}&\geq 0\quad\forall k,l\in\setval{1}{n}, k\neq a, l\neq b
	\end{align*}
	Its dual has the form
	\begin{align*}
		\min \sum_{\substack{l=1\\l\neq b}}^n\beta^1_{lab} + \sum_{\substack{k=1\\k\neq a}}^n\beta^2_{kab} &\\
		\text{s.t.}\quad\beta^1_{lab} + \beta^2_{kab}&\geq \frac{p_{ak}d_{bl}}{2}\quad\forall k,l\in\setval{1}{n}, k\neq a, l\neq b
	\end{align*}
	The primal problem is a linear assignment problem with total unimodular matrix, thus there exists an optimal integral solution with value (by definition)~$\frac{u_{ab}}{2}$. The corresponding optimal dual solution therefore satisfies
	\begin{align*}
		\sum_{\substack{l=1\\l\neq b}}^n\beta^1_{lab} + \sum_{\substack{k=1\\k\neq a}}^n\beta^2_{kab} & = \frac{u_{ab}}{2}\\
		\quad\beta^1_{lab} + \beta^2_{kab}&\geq \frac{p_{ak}d_{bl}}{2}\quad\forall k,l\in\setval{1}{n}, k\neq a, l\neq b
	\end{align*}
	For $d\neq b$ 
	\begin{align*}
		\beta^1_{lad} &= 
		\begin{cases}
			\beta^1_{lab}&\text{if }l\neq b\\
			\beta^1_{dab}&\text{if }l = b
		\end{cases}\quad\forall l\in\setval{1}{n},l\neq d\\
		\beta^2_{kad} &= \beta^2_{kab}\quad\forall k\in\setval{1}{n},k\neq a
	\end{align*}
	yields the required coefficients, for $c\neq a$
	\begin{align*}
		\beta^1_{lcb} &=\beta^1_{lab}\quad\forall l\in\setval{1}{n}l\neq b\\
		\beta^2_{kcb} &=
		\begin{cases}
			\beta^2_{kab}&\text{if }k\neq a\\
			\beta^2_{cab}&\text{if }k = a
		\end{cases}\quad\forall k\in\setval{1}{n},k\neq c
	\end{align*}
	is a valid choice.
\end{proof}

\begin{proof}[Thm.~\ref{theorem:AJbesserXY}]
	As both defining inequalities of the linearization of Xia and Yuan are contained in the projection of the linearization of Adams and Johnson by Prop.~\ref{prop:inequalityLContained} and \ref{prop:inequalityUContained}, the latter is contained in the former.
\end{proof}

To sharpen the result of Thm.~\ref{theorem:AJbesserXY} that the projected polytope of the relaxed linearization of Adams and Johnson is contained in the one of Xia and Yuan, we provide proof that the formulations are not equal.

\begin{theorem}
\label{theorem:AJwirklichbesserXY}
	Let $D, P\in\Rbb^{n\times n}$ coefficient matrices, and denote by $LXY(D,P)$ the linear relaxation of the Xia-Yuan linearization, and by $LAJ(D,P)$ the linear relaxation of the Adams-Johnson linearization. Recall the definition of~$l_{ij}$ in \eqref{eq:DefinitionL}
	\begin{equation*}
		l_{ij} = \min_{x\in X_{n-1}}\sum_{\substack{k,l=1\\k\neq i, l\neq j}}^np_{ik}d_{jl}x_{kl}\quad\forall i,j\in\setval{1}{n}
	\end{equation*}
	and denote by $x\in\argmin(l_{ij})$ a vector $x\in X_{n-1}$ minimizing the expression of $l_{ij}$.
	
	Assume coefficients $D, P$ have a structure such that there exist $a,b,c,d \in\setval{1}{n},a\neq c, b\neq d$ with
	\begin{enumerate}
		\item $x_{cd}=1\;\forall x\in\argmin(l_{ab})$ and
		\item $x_{ab}=0\;\forall x\in\argmin(l_{cd})$.
	\end{enumerate}
	
	Then
	\begin{equation*}
		\Proj_{x,z}\left(LAJ(D,P)^+\right) \subsetneq LXY(D,P).
	\end{equation*}
\end{theorem}
\begin{proof}
	Let $D, P\in\Rbb^{n\times n}$ coefficient matrices such that there exist $a,b,c,d \in\setval{1}{n},a\neq c, b\neq d$ with the given assumptions. Consider the solution $(x^\star, z^\star)\in LXY(D,P)$ where $x^\star_{ij}=\frac{1}{n}\;\forall i,j\in\setval{1}{n}$ and $z^\star_{ij} = l_{ij}x^\star_{ij}\;\forall i,j\in\setval{1}{n}$ which is valid for $LXY(D,P)$. For a corresponding solution $(x^\star, y^\star)\in LAJ(D,P)$ that is to be projected to $(x^\star, z^\star)$, the affine transformation \eqref{eq:zyzusammenhang} demands
	\begin{equation*}
		\frac{1}{n}l_{ij} = \sum_{\substack{k,l=1\\k\neq i, l\neq j}}^np_{ik}d_{jl}y^\star_{ijkl}\quad\forall i,j\in\setval{1}{n},
	\end{equation*}
	in other words
	\begin{equation*}
		(y^\star_{ijkl})_{k,l \in\setval{1}{n}}\in\argmin(\frac{1}{n}l_{ij})\quad\forall i,j\in\setval{1}{n}.
	\end{equation*}
	Considering pairs $(a,b)$ and $(c,d)$ and the assumptions, we conclude $y^\star_{abcd} = 1$ and $y^\star_{cdab} = 0$,	which contradicts $(x^\star, y^\star)\in LAJ(D,P)$ due to \eqref{eq:aj4}.
\end{proof}

We see that the linearization of Xia and Yuan tries to represent the quadratic term of the objective function by fixing each assignment for itself and then estimating the impact on the remaining assignment. This, however, totally neglects the symmetry within the problem, which is still handled by the linearization of Adams and Johnson at the price of its $\Ocal(n^4)$ variables.

\begin{remark}
	The requirements of Thm.~\ref{theorem:AJwirklichbesserXY} are not strict. In fact, if they do not apply, the QAP is easy to solve as the solution space is reduced to $n$ possible solutions given by the assumptions, one of which is optimal.
\end{remark}

\begin{example}[Projected linearization of Adams and Johnson strictly contained in linearization of Xia and Yuan]
\label{example:strictly_contained}
	Let $n = 3$ and consider coefficient matrices
	\begin{align*}
		D &=
		\begin{pmatrix}
			0&1&2\\
			1&0&1\\
			2&1&0
		\end{pmatrix},&
		P &= \frac{1}{18}
		\begin{pmatrix}
			0&4&2\\
			3&0&3\\
			4&2&0
		\end{pmatrix}
	\end{align*}
	with 
	\begin{align*}
		l &= \frac{1}{18}
		\begin{pmatrix}
			9&6&9\\
			8&6&8\\
			8&6&8
		\end{pmatrix},&
		\argmin(l_{11})&=
		\set{\begin{pmatrix}
			1&0\\
			0&1
		\end{pmatrix}},&
		\argmin(l_{33})&=
		\set{\begin{pmatrix}
			0&1\\
			1&0
		\end{pmatrix}}.		
	\end{align*}
	
	The only vector $(x^\star, y^\star)$ to be projected onto $(x^\star, z^\star)\in LXY(D,P)$ where $x^\star_{ij} = \frac{1}{3}\;\forall i,j\in\setval{1}{3}$ and $z^\star_{ij} = l_{ij}x^\star_{ij}\;\forall i,j\in\setval{1}{3}$ has
	\begin{align*}
		(y^\star_{11kl})_{k,l\in\set{2,3}} &= \frac{1}{3}
		\begin{pmatrix}
			1&0\\
			0&1			
		\end{pmatrix},&
		(y^\star_{33kl})_{k,l\in\set{1,2}} &= \frac{1}{3}
		\begin{pmatrix}
			0&1\\
			1&0
		\end{pmatrix}
	\end{align*}
	and is not contained in $LAJ(D,P)$ as $y^\star_{1133}\neq y^\star_{3311}$.
\end{example}

\section{$ab$-cuts for the linearization of Xia and Yuan}
\label{sec:newInequalities}

One main questions arising from Thm.~\ref{theorem:AJbesserXY} and~\ref{theorem:AJwirklichbesserXY} is whether one can use the result to improve the solution performance of the linearization of Xia and Yuan. In practical experiments using a Branch and Bound approach, this linearization tends to provide a small formulation with fast solution times of each linear program in the search tree; however its weak linear relaxation yields many nodes to search. It is therefore not superior to the linearization of Adams and Johnson, which exhibits a bigger formulation and resulting slower solution of the linear programs; however this effect is countermanded by the reduced number of nodes to solve which originates in its better linear bound. In a context with limited computation time or with limited memory resources, an ILP formulation with reduced size that is capable to quickly evaluate many nodes with little memory consumption can be of interest.

In the following, we will therefore investigate whether the projected inequalities of the linearization of Adams and Johnson can be used to strengthen the linear relaxation of the linearization of Xia and Yuan. Let in the following $(x^\star, z^\star) \in LXY(D,P)$ denote values of the linear relaxation of Xia and Yuan. The separation problem for given $(x^\star, z^\star)\in LXY(D,P)$
\begin{equation*}
	\max_{(\alpha,\beta^1,\beta^2,\gamma)\in W}\sum_{i,j=1}^n\alpha_{ij}z^\star_{ij} - \sum_{i,j=1}^n \left(\sum_{l=1}^n \beta^1_{lij} + \sum_{k=1}^n\beta^2_{kij}\right)x^\star_{ij} 
\end{equation*}
is equivalent to verifying whether there exist $y_{ijkl}\in\Rbb^{n^4}$ such that the extended formulation of Adams and Johnson $LAJ(D,P)^+|_{x=x^\star,z=z^\star}$ with fixed variables $x, z$ is non-empty. If there exist $y_{ijkl}$, the separation problem's objective value is not greater than 0. If on the other hand there are no such variables $y_{ijkl}$, the separation problem is unbounded (\ie there is a violated inequality, and as $W$ is a cone the inequality can be scaled arbitrarily). 

As the separation on the complete projection cone therefore is equivalent to computing the solution of the linearization of Adams and Johnson, we turn to the investigation of subclasses of inequalities defined by the projection. A suitable class of inequalities with a corresponding efficient separation method might provide a means to strengthen the linear relaxation without incorporating the complete complexity of the linearization of Adams and Johnson. One class of inequalities reveals itself as a generalization of \eqref{eq:XYlInequality}, derived from the idea used in the proof of Prop.~\ref{prop:inequalityLContained}.
\begin{theorem}[$ab$-cuts]
\label{theorem:cuts}
	Let $a,b\in\setval{1}{n}$ and $(x^\star,z^\star)\in LXY(D,P)$ a solution of the relaxed linearization of Xia and Yuan. Consider a (not necessarily optimal) dual solution $(\beta^1_{lab}, \beta^2_{kab}, \delta_{kl})\in\Rbb^{2(n-1)+(n-1)^2}$ to
	\begin{align*}
		sep = \min \sum_{\substack{k,l=1\\k\neq a, l\neq b}}^n p_{ak}d_{bl}x_{kl}&\\
		\text{s.t.}\quad-\sum_{\substack{k=1\\k\neq a}}^nx_{kl}&=-x^\star_{ab}\quad\forall l\in\setval{1}{n}, l\neq b &(\beta^1_{lab})\\
		-\sum_{\substack{l=1\\l\neq b}}^nx_{kl}&=-x^\star_{ab}\quad\forall k\in\setval{1}{n}, k\neq a &(\beta^2_{kab})\\
		x_{kl}&\geq 0\quad\forall k,l\in\setval{1}{n}, k\neq a, l\neq b\\
		-x_{kl}&\geq -x^\star_{kl}\quad\forall k,l\in\setval{1}{n}, k\neq a, l\neq b&(\delta_{kl})
	\end{align*}
	We denote the dual objective value of $(\beta^1_{lab}, \beta^2_{kab}, \delta_{kl})$ by $sep_{dual}(\beta^1_{lab}, \beta^2_{kab}, \delta_{kl})$, and the optimal primal objective value by $sep$.
	
	Then $ab$-cut
	\begin{equation}
		\label{eq:newInequality}
		z_{ab}\geq -\left(\sum_{\substack{l=1\\l\neq b}}^n\beta^1_{lab} + \sum_{\substack{k=1\\k\neq a}}^n\beta^2_{kab}\right)x_{ab} - \sum_{\substack{k,l=1\\k\neq a, l\neq b}}^n \delta_{kl}x_{kl}\tag{$ab$-cut}
	\end{equation}
	is valid for the QAP. 
	
	Furthermore, the $ab$-cut is violated by $(x^\star,z^\star)$ if $sep_{dual}(\beta^1_{lab}, \beta^2_{kab}, \delta_{kl}) > z^\star_{ab}$, and there is no violated inequality if $sep \leq z^\star_{ab}$.
\end{theorem}

\begin{proof}
	Let $a,b\in\setval{1}{n}$, $(x^\star,z^\star)\in LXY(D,P)$ a valid point to the linearization and consider a dual solution $(\beta^1_{lab}, \beta^2_{kab}, \delta_{kl})\in\Rbb^{2(n-1)+(n-1)^2}$ to 
	\begin{align*}
		sep = \min \sum_{\substack{k,l=1\\k\neq a, l\neq b}}^n p_{ak}d_{bl}x_{kl}&\\
		\text{s.t.}\quad-\sum_{\substack{k=1\\k\neq a}}^nx_{kl}&=-x^\star_{ab}\quad\forall l\in\setval{1}{n}, l\neq b &(\beta^1_{lab})\\
		-\sum_{\substack{l=1\\l\neq b}}^nx_{kl}&=-x^\star_{ab}\quad\forall k\in\setval{1}{n}, k\neq a &(\beta^2_{kab})\\
		x_{kl}&\geq 0\quad\forall k,l\in\setval{1}{n}, k\neq a, l\neq b\\
		-x_{kl}&\geq -x^\star_{kl}\quad\forall k,l\in\setval{1}{n}, k\neq a, l\neq b&(\delta_{kl})
	\end{align*}
	The dual problem is
	\begin{align*}
		\max -x^\star_{ab}\left(\sum_{\substack{l=1\\l\neq b}}^n\beta^1_{lab} + \sum_{\substack{k=1\\k\neq a}}^n\beta^2_{kab}\right) &- \sum_{\substack{k,l=1\\k\neq a, l\neq b}}x^\star_{kl}\delta_{kl}\\
		\text{s.t.}\quad-\beta^1_{lab} - \beta^2_{kab}&\leq p_{ak}d_{bl}+\delta_{kl}\quad\forall k,l\in\setval{1}{n}, k\neq a, l\neq b\\
		\delta_{kl}&\geq 0 \quad\forall k,l\in\setval{1}{n}, k\neq a, l\neq b
	\end{align*}
	By combining the dual solution $\beta^1_{lab}, \beta^2_{kab}, \delta_{kl}$ with the following multipliers
	\begin{itemize}
		\item $\alpha_{ij} = 
			\begin{cases}
				-1&\text{if }(i,j) = (a,b)\\
				0&\text{otherwise}
			\end{cases}
			\quad\forall i,j \in\setval{1}{n}$
		\item $\beta^1_{lij} = 0\quad\forall l,i,j\in\setval{1}{n}, l\neq j, l\neq b, (i,j) \neq (a,b)$
		\item $\beta^2_{kij} = 0\quad\forall k,i,j\in\setval{1}{n}, k\neq a, k \neq a, (i,j) \neq (a,b)$
		\item $\beta^1_{lij} + \beta^2_{kij} = \delta_{ij}\quad\forall i,j\in\setval{1}{n}, i\neq a, j\neq b, (k,l) = (a,b)$
		\item $\gamma_{ijkl} = 
			\begin{cases}
				-p_{ak}d_{bl} - \delta_{kl}&\text{if }(i,j) = (a,b)\\
				p_{ai}d_{bj} + \delta_{ij}&\text{if }(k,l) = (a,b)\\
				0&\text{otherwise}
			\end{cases}
			\quad
			\begin{aligned}
				\forall i,j,k,l\in\setval{1}{n}\\
				k\neq i, l\neq j
			\end{aligned}
			$
		\item $\beta^1_{lab}, \beta^2_{kab}$ as above in the dual solution
	\end{itemize}
	we obtain a collection $(\alpha,\beta^1,\beta^2,\gamma)$ that is contained in the projection cone~$W$ as shown by verification of \eqref{eq:proj_cone:1}:
	\begin{enumerate}
		\item Let $(i,j)=(a,b), k\neq a, l\neq b$:
			\begin{equation*}
				p_{ak}d_{bl}\alpha_{ab} = -p_{ak}d_{bl} = \delta_{kl} -p_{ak}d_{bl} - \delta_{kl} = \beta^1_{bkl} + \beta^2_{akl} + \gamma_{abkl}
			\end{equation*}
		\item Let $(k,l) = (a,b), i\neq a, j\neq b$:
			\begin{equation*}
				p_{ia}d_{jb}\alpha_{ij} = 0 = -p_{ai}d_{bj} - \delta_{ij} + p_{ai}d_{bj} + \delta_{ij} \leq \beta^1_{jab} + \beta^2_{iab} + \gamma_{ijab}
			\end{equation*}
		\item Let $i\neq a, j\neq b, k\neq a, l\neq b$:
			\begin{equation*}
				p_{ik}d_{jl}\alpha_{ij} = 0 \leq 0 + 0 + 0 = \beta^1_{jkl} + \beta^2_{ikl} + \gamma_{ijkl}
			\end{equation*}
	\end{enumerate}
	Equation \eqref{eq:proj_cone:2} of the projection is trivially valid. The resulting projected inequality is equivalent to the $ab$-cut
	\begin{align*}
		\sum_{i,j=1}^n\alpha_{ij}z_{ij} - \sum_{i,j=1}^n \left(\sum_{\substack{l=1\\l\neq j}}^n \beta^1_{lij} + \sum_{\substack{k=1\\k\neq i}}^n\beta^2_{kij}\right)x_{ij} & =\\
		-z_{ab} - \left(\sum_{\substack{l=1\\l\neq b}}^n \beta^1_{lab} + \sum_{\substack{k=1\\k\neq a}}^n\beta^2_{kab}\right)x_{ab} - \sum_{\substack{i,j=1\\i\neq a, j\neq b}}^n\left(\beta^1_{bij}+\beta^2_{aij}\right)x_{ij}& =\\
		-z_{ab} - \left(\sum_{\substack{l=1\\l\neq b}}^n \beta^1_{lab} + \sum_{\substack{k=1\\k\neq a}}^n\beta^2_{kab}\right)x_{ab} - \sum_{\substack{i,j=1\\i\neq a, j\neq b}}^n\delta_{ij}x_{ij}& \leq 0.
	\end{align*}
	
	Last, as the inequality equals $-z_{ab} + sep_{dual}(\beta^1_{lab}, \beta^2_{kab}, \delta_{kl}) \leq 0$ it is violated if $sep_{dual}(\beta^1_{lab}, \beta^2_{kab}, \delta_{kl}) > z^\star_{ab}$, and by duality there is no violated inequality if $sep \leq z^\star_{ab}$.
\end{proof}

\begin{remark}
	$ab$-cuts can be viewed as a generalization of \eqref{eq:XYlInequality} both from their proof as well as their interpretation. 
	
	Restricting the dual variable $\delta_{kl}=0\;\forall k,l\in\setval{1}{n}$ in the proof of $ab$-cuts would result in \eqref{eq:XYlInequality}. In general, the multipliers collected in the proof are similar to the ones used in Prop.~\ref{prop:inequalityLContained}, but allow for some variation by means of the $\delta_{kl}$ without searching the complete projection cone $W$.
	
	Using a problem-specific view, $ab$-cuts strengthen \eqref{eq:XYlInequality} by including information about the current fractional solution through the constraint $x_{kl}\leq x^\star_{kl}\;\forall k,l\in\setval{1}{n}, k\neq a, l\neq b$ of the separation problem. Therefore the lower bound on $z_{ab}$ in dependence of $x_{ab}$ reflects the structure of the current fractional solution, \eg,  by excluding fractional assignments where $x^\star_{kl} = 0$ (which obviously is not the case for $l_{ab}$ used in \eqref{eq:XYlInequality}).
\end{remark}

Thm.~\ref{theorem:cuts} states conditions under which a violated $ab$-cut exists, however these depend on the objective value of the separation problem. In the following, we add an a priori requirement that can be checked before solving the separation problem.

\begin{proposition}
\label{prop:cond_on_cuts}
	Let $a,b\in\setval{1}{n}$ and $(x^\star,z^\star)\in LXY(D,P)$ a solution of the relaxed linearization of Xia and Yuan.
	
	There is no $ab$-cut that is violated by $(x^\star, z^\star)$ if any of the following conditions is true:
	\begin{enumerate}
		\item $x^\star_{ab} = 0$, or
		\item $x^\star_{ab} > 0$ and
			\begin{equation*}
				\exists \bar{x}\in\argmin(l_{ab}): \bar{x}_{kl} \leq \frac{x^\star_{kl}}{x^\star_{ab}}\quad\forall k,l\in\setval{1}{n}, k\neq a, l\neq b.
			\end{equation*}
	\end{enumerate}
\end{proposition}
\begin{proof}
	Let $a,b\in\setval{1}{n}$, $(x^\star,z^\star)\in LXY(D,P)$ and $\bar{x}$ as defined. 
	
	If $x^\star_{ab} = 0$ the only valid solution of the separation problem is $x_{kl}=0\;\forall k,l\in\setval{1}{n}, k\neq a, l\neq b$ with objective value $sep = 0 \leq z^\star_{ab}$, therefore no violated $ab$-cut exists.
	
	If $x^\star_{ab} > 0$, then $\tilde{x} = x^\star_{ab}\bar{x}\in x^\star_{ab}X_{n-1}$ is a scaled permutation matrix, and due to the assumptions a valid element of the separation problem with objective value $x^\star_{ab}l_{ab}$. We conclude using \eqref{eq:XYlInequality} that
	\begin{equation*}
		z^\star_{ab}\geq l_{ab}x^\star_{ab} \geq sep,
	\end{equation*}
	wherefore no violated $ab$-cut exists.
\end{proof}

No further polytopal properties of $ab$-cuts are known to date, especially questions whether they are facet-defining or can be lifted further require additional analysis of the polytope. 

\begin{example}[Linearization of Xia and Yuan with cuts yields integral solution]
\label{example:integral_sol_with_cuts}
	Let $n=3$ and consider $D,P$ as defined in Example~\ref{example:strictly_contained}. 
	
	When solving $LXY(D,P)$ we obtain the first fractional solution 
	\begin{align*}
		x^\star &= \frac{1}{4}
		\begin{pmatrix}
			3&0&1\\
			1&3&0\\
			0&1&3
		\end{pmatrix},&
		z^\star&=\frac{1}{72}
		\begin{pmatrix}
			24&0&8\\
			9&18&0\\
			0&6&24
		\end{pmatrix}.
	\end{align*}
	This solution violates $ab$-cuts
	\begin{align*}
		z_{13}&\geq \frac{5}{9}x_{13} - \frac{3}{9}x_{31},&z_{33}&\geq \frac{5}{9}x_{33} - \frac{3}{9}x_{12}.
	\end{align*}
	The latter cut extends $z_{33}\geq l_{33}x_{33}$, as the value of $l_{33}$ assumes $x_{12}=1$ (compare Example~\ref{example:strictly_contained}). In the fractional solution, however, $x^\star_{12}=0$. Therefore, the bound on $z_{33}$ can be strengthened when $x_{12} = 0$, which is done in the $ab$-cut.
	
	After resolving the problem with these cuts, we obtain as second fractional solution
	\begin{align*}
		x^\star &= \frac{1}{5}
		\begin{pmatrix}
			4&1&0\\
			1&2&2\\
			0&2&3
		\end{pmatrix},&
		z^\star&=\frac{1}{90}
		\begin {pmatrix}
			32&6&0\\
			9&12&18\\
			0&12&24
		\end{pmatrix}.
	\end{align*}
	This solution violates $ab$-cuts
	\begin{align*}
		z_{11}&\geq \frac{5}{9}x_{11} - \frac{1}{9}x_{22},& z_{33}&\geq \frac{6}{9}x_{33} - \frac{2}{9}x_{21} - \frac{1}{9}x_{22}.
	\end{align*}	
	Adding these cuts results in a third fractional solution
	\begin{align*}
		x^\star &= \frac{1}{4}
		\begin{pmatrix}
			2&1&1\\
			2&2&0\\
			0&1&3
		\end{pmatrix},&
		z^\star&=\frac{1}{36}
		\begin {pmatrix}
			8&3&5\\
			9&6&0\\
			0&3&12
		\end{pmatrix}.
	\end{align*}
	This solution violates $ab$-cut
	\begin{align*}
		z_{33}&\geq \frac{4}{9}x_{33} - \frac{2}{9}x_{12} - \frac{1}{9}x_{22}.
	\end{align*}	
	After adding this cut, we obtain an optimal integral solution
	\begin{align*}
		x^\star &= 
		\begin{pmatrix}
			0&1&0\\
			1&0&0\\
			0&0&1
		\end{pmatrix},&
		z^\star&=\frac{1}{18}
		\begin {pmatrix}
			0&6&0\\
			9&0&0\\
			0&0&8
		\end{pmatrix}.
	\end{align*}
\end{example}

Example~\ref{example:integral_sol_with_cuts} shows that the inequalities defined in Thm.~\ref{theorem:cuts} may suffice to ensure an integral solution to the relaxed linearization of Xia and Yuan. However, this is not true in general.

	%
	%

\begin{example}[Linearization of Xia and Yuan with cuts does not yield integral solution]
\label{example:no_integral_sol_with_cuts}
	Let $n=4$ and consider coefficient matrices
	\begin{align*}
		D &=
		\begin{pmatrix}
			0&1&2&3\\
			1&0&1&2\\
			2&1&0&1\\
			3&2&1&0
		\end{pmatrix}&
		P = \frac{1}{16}
		\begin{pmatrix}
			0&2&1&1\\
			2&0&2&0\\
			1&1&0&2\\
			2&1&1&0
		\end{pmatrix}.
	\end{align*}
	
	When solving the linearization of Xia and Yuan with this input, we find 21 $ab$-cuts in 4 iterations and obtain the fractional solution
	\begin{align*}
		x^\star &= \frac{1}{2}
		\begin{pmatrix}
			1&1&0&0\\
			1&1&0&0\\
			0&0&1&1\\
			0&0&1&1
		\end{pmatrix},&
		z^\star&=\frac{1}{32}
		\begin {pmatrix}
			7&5&0&0\\
			6&4&0&0\\
			0&0&5&7\\
			0&0&5&8
		\end{pmatrix}.
	\end{align*}
	No further $ab$-cuts violated by this fractional solution can be found.
	
	Note furthermore that $(a,b) = (4,4)$ does not meet the requirements of Prop.~\ref{prop:cond_on_cuts}, yet no violated cuts can be found. For all other combinations of $a,b \in\setval{1}{n}$, the requirements of Prop.~\ref{prop:cond_on_cuts} apply.
\end{example}

The separation of $ab$-cuts can be done by solving the respective matching problem. When used repeatedly in a Branch and Bound approach, one could make use of advanced linear programing techniques such as a warmstart of the separation problem as the dual polytope remains unchanged, only the dual objective function changes.

%
%

\section{Outlook}
\label{sec:outlook}

Using a Lift-And-Project approach, we proved that the linearization of Adams and Johnson provides a tighter linear bound than the linearization of Xia and Yuan. To be precise, its polytope is strictly contained in the one of Xia and Yuan. Subsequently, we used the theoretical insights obtained from the proof to derive a new family of cuts (called $ab$-cuts) that can be used in a Branch and Bound technique to strengthen the linearization of Xia and Yuan. We gave an interpretation of the cuts and their separation, as well as conditions under which no violated $ab$-cuts exist given a fractional solution.

Future work includes practical tests of the Branch and Cut approach, an extended research of the coefficients included in $ab$-cuts (an LP solver returns extremal dual variables, which might not be the best choice) as well as the development of further cut families. Preliminary results show that the $ab$-cuts reduce the gap between linear relaxation and optimal solution by approximately 10\%.

\bibliographystyle{spmpsci}      
\bibliography{literatur}   

\begin{thebibliography}{10}
\providecommand{\url}[1]{{#1}}
\providecommand{\urlprefix}{URL }
\expandafter\ifx\csname urlstyle\endcsname\relax
  \providecommand{\doi}[1]{DOI~\discretionary{}{}{}#1}\else
  \providecommand{\doi}{DOI~\discretionary{}{}{}\begingroup
  \urlstyle{rm}\Url}\fi

\bibitem{MR1924955}
Maia~de Abreu, N.M., Boaventura~Netto, P.O., Maia~Querido, T., Ferreira~Gouvea,
  E.: Classes of quadratic assignment problem instances: isomorphism and
  difficulty measure using a statistical approach.
\newblock Discrete Appl. Math. \textbf{124}(1-3), 103--116 (2002).
\newblock \doi{10.1016/S0166-218X(01)00333-X}.
\newblock \urlprefix\url{https://doi.org/10.1016/S0166-218X(01)00333-X}.
\newblock Workshop on Discrete Optimization (Piscataway, NJ, 1999)

\bibitem{MR2314426}
Adams, W.P., Guignard, M., Hahn, P.M., Hightower, W.L.: A level-2
  reformulation-linearization technique bound for the quadratic assignment
  problem.
\newblock European J. Oper. Res. \textbf{180}(3), 983--996 (2007).
\newblock \doi{10.1016/j.ejor.2006.03.051}.
\newblock \urlprefix\url{http://dx.doi.org/10.1016/j.ejor.2006.03.051}

\bibitem{MR1290346}
Adams, W.P., Johnson, T.A.: Improved linear programming-based lower bounds for
  the quadratic assignment problem.
\newblock In: Quadratic assignment and related problems ({N}ew {B}runswick,
  {NJ}, 1993), \emph{DIMACS Ser. Discrete Math. Theoret. Comput. Sci.},
  vol.~16, pp. 43--75. Amer. Math. Soc., Providence, RI (1994)

\bibitem{MR2194735}
Balas, E.: Projection, lifting and extended formulation in integer and
  combinatorial optimization.
\newblock Ann. Oper. Res. \textbf{140}, 125--161 (2005).
\newblock \doi{10.1007/s10479-005-3969-1}.
\newblock \urlprefix\url{http://dx.doi.org/10.1007/s10479-005-3969-1}

\bibitem{MR2979433}
Fischetti, M., Monaci, M., Salvagnin, D.: Three ideas for the quadratic
  assignment problem.
\newblock Oper. Res. \textbf{60}(4), 954--964 (2012).
\newblock \doi{10.1287/opre.1120.1073}.
\newblock \urlprefix\url{http://dx.doi.org/10.1287/opre.1120.1073}

\bibitem{MR678819}
Frieze, A.M., Yadegar, J.: On the quadratic assignment problem.
\newblock Discrete Appl. Math. \textbf{5}(1), 89--98 (1983).
\newblock \doi{10.1016/0166-218X(83)90018-5}.
\newblock \urlprefix\url{https://doi.org/10.1016/0166-218X(83)90018-5}

\bibitem{KAUFMAN1978207}
Kaufman, L., Broeckx, F.: An algorithm for the quadratic assignment problem
  using bender's decomposition.
\newblock European Journal of Operational Research \textbf{2}(3), 207 -- 211
  (1978).
\newblock \doi{https://doi.org/10.1016/0377-2217(78)90095-4}.
\newblock
  \urlprefix\url{http://www.sciencedirect.com/science/article/pii/0377221778900954}

\bibitem{MR0089106}
Koopmans, T.C., Beckmann, M.: Assignment problems and the location of economic
  activities.
\newblock Econometrica \textbf{25}, 53--76 (1957).
\newblock \doi{10.2307/1907742}.
\newblock \urlprefix\url{https://doi.org/10.2307/1907742}

\bibitem{MR0152361}
Lawler, E.L.: The quadratic assignment problem.
\newblock Management Sci. \textbf{9}, 586--599 (1962/1963).
\newblock \doi{10.1287/mnsc.9.4.586}.
\newblock \urlprefix\url{https://doi.org/10.1287/mnsc.9.4.586}

\bibitem{MR2267435}
Loiola, E.M., Maia~de Abreu, N.M., Boaventura-Netto, P.O., Hahn, P., Querido,
  T.: A survey for the quadratic assignment problem.
\newblock European J. Oper. Res. \textbf{176}(2), 657--690 (2007).
\newblock \doi{10.1016/j.ejor.2005.09.032}.
\newblock \urlprefix\url{https://doi.org/10.1016/j.ejor.2005.09.032}

\bibitem{nyberg2014some}
Nyberg, A., et~al.: Some reformulations for the quadratic assignment problem.
\newblock Ph.D. thesis (2014)

\bibitem{MR2913063}
Wright, S.E.: New linearizations of quadratic assignment problems.
\newblock Comput. Oper. Res. \textbf{39}(11), 2858--2866 (2012).
\newblock \doi{10.1016/j.cor.2012.02.017}.
\newblock \urlprefix\url{http://dx.doi.org/10.1016/j.cor.2012.02.017}

\bibitem{MR2373098}
Xia, Y.: Gilmore-{L}awler bound of quadratic assignment problem.
\newblock Front. Math. China \textbf{3}(1), 109--118 (2008).
\newblock \doi{10.1007/s11464-008-0010-4}.
\newblock \urlprefix\url{http://dx.doi.org/10.1007/s11464-008-0010-4}

\bibitem{MR2238659}
Xia, Y., Yuan, Y.X.: A new linearization method for quadratic assignment
  problems.
\newblock Optim. Methods Softw. \textbf{21}(5), 805--818 (2006).
\newblock \doi{10.1080/10556780500273077}.
\newblock \urlprefix\url{http://dx.doi.org/10.1080/10556780500273077}

\bibitem{MR3084089}
Zhang, H., Beltran-Royo, C., Ma, L.: Solving the quadratic assignment problem
  by means of general purpose mixed integer linear programming solvers.
\newblock Ann. Oper. Res. \textbf{207}, 261--278 (2013).
\newblock \doi{10.1007/s10479-012-1079-4}.
\newblock \urlprefix\url{http://dx.doi.org/10.1007/s10479-012-1079-4}

\end{thebibliography}


\end{document}